\documentclass[oneside,11pt]{amsart}

\usepackage{amsmath, amssymb, amsthm, latexsym, enumerate, mathrsfs, textcomp, bbm}
\usepackage{graphicx}
\usepackage{amsmath}
\usepackage[T1]{fontenc}
\usepackage{cite}
\usepackage[latin1]{inputenc}
\usepackage{cite}
\usepackage[small]{caption}

\usepackage{tikz,tikz-cd}
\usetikzlibrary{arrows,chains,matrix,positioning,scopes}

\makeatletter
\tikzset{join/.code=\tikzset{after node path={%
\ifx\tikzchainprevious\pgfutil@empty\else(\tikzchainprevious)%
edge[every join]#1(\tikzchaincurrent)\fi}}}
\makeatother
\tikzset{>=stealth',every on chain/.append style={join},
         every join/.style={->}}
\tikzstyle{labeled}=[execute at begin node=$\scriptstyle,
   execute at end node=$]

\usepackage{amsfonts}
\usepackage{amssymb}
\usepackage{amsxtra}

\theoremstyle{plain}
\newtheorem{theorem}{Theorem}[section]
\newtheorem{corollary}[theorem]{Corollary}
\newtheorem{lemma}[theorem]{Lemma}
\newtheorem{proposition}[theorem]{Proposition}
\newtheorem{problem}[theorem]{Problem}
\newtheorem{conjecture}[theorem]{Conjecture}

\theoremstyle{definition}

\newtheorem*{remark}{Remark}

\theoremstyle{definition}

\usepackage{ifthen}

\newcommand{\showcomments}{yes}

\newsavebox{\commentbox}
{\ifthenelse{\equal{\showcomments}{yes}}%
{\footnotemark
    \begin{lrbox}{\commentbox}
    \begin{minipage}[t]{1.25in}\raggedright\sffamily\upshape\tiny
    \footnotemark[\arabic{footnote}]}
{\begin{lrbox}{\commentbox}}}%
{\ifthenelse{\equal{\showcomments}{yes}}%
{\end{minipage}\end{lrbox}\marginpar{\usebox{\commentbox}}}
{\end{lrbox}}}

\newcommand{\spc}{\hspace{1 mm}}

\newcommand{\bbr}{\mathbb{R}}
\newcommand{\bbn}{\mathbb{N}}

\newcommand{\bbz}{\mathbb{Z}}

\newcommand{\XT}{\tilde{X}}

\newcommand{\xt}{\tilde{x}}
\newcommand{\yt}{\tilde{y}}

\newcommand{\bbp}{\mathbb{P}}
\newcommand{\bndry}{\partial}

\newcommand{\C}{\mathcal{C}}
\newcommand{\relbndry}{\bndry (\Gamma,\bbp)}

\newcommand{\R}{\mathcal{R}}
\newcommand{\N}{\mathcal{N}}
\newcommand{\overB}{\overline{B}}
\newcommand{\F}{\mathcal{F}}
\newcommand{\D}{\mathcal{D}}
\newcommand{\lacton}{\curvearrowright}

\DeclareMathOperator{\isom}{Isom}
\DeclareMathOperator{\CAT}{CAT}

\DeclareMathOperator{\stab}{Stab}

\DeclareMathOperator{\diam}{diam}

\DeclareMathOperator{\orb}{Orb}

\newcommand{\presentation}[2]{\langle\, {#1} \mid {#2} \,\rangle}

\newcommand{\bigset}[2]{ \bigl\{ \, {#1} \bigm| {#2} \, \bigr\} }

\newcommand{\Swiatkowski}{{\'{S}}wi{\k{a}}tkowski}

\begin{document}

\title[Boundary classification and 2-ended splittings]{%
Boundary classification and 2-ended splittings of groups with isolated flats}

\author[M.~Haulmark]{Matthew Haulmark}
\address{Department of Mathematics\\
1326 Stevenson Center\\
Vanderbilt University\\
Nashville, TN 37240 USA}
\email{m.haulmark@vanderbilt.edu}

\begin{abstract}
In this paper we provide a classification theorem for 1-dimensional boundaries of groups with isolated flats. Given a group $\Gamma$ acting geometrically on a $\CAT(0)$ space $X$ with isolated flats and 1-dimensional boundary, we show that if $\Gamma$ does not split over a virtually cyclic subgroup, then $\bndry X$ is homeomorphic to a circle, a Sierpinski carpet, or a Menger curve. This theorem generalizes a theorem of Kapovich-Kleiner, and resolves a question due to Kim Ruane.

We also study the relationship between local cut points in $\bndry X$ and splittings of $\Gamma$ over $2$-ended subgroups. In particular, we generalize a theorem of  Bowditch by showing that the existence of a local point in $\bndry X$ implies that $\Gamma$ splits over a $2$-ended subgroup.
\end{abstract}

\keywords{ $\CAT(0)$, Isolated Flats, Group Boundaries, JSJ Decompositions, Splittings, Ends of Spaces}

\date{\today}

\maketitle

\section{Introduction}
\label{sec:Introduction}

When a group $\Gamma$ acts discretely on a geometric space $X$, we can often compactify $X$ by attaching a ``boundary at infinity'' $\bndry X$ to $X$. In the presence of non-positive curvature, $\Gamma$ has an induced action by homeomorphisms on the boundary. There are strong connections between the topological properties of $\bndry X$ and the algebraic properties of $\Gamma$. A natural question posed by Kapovich and Kleiner \cite{KK} is: which topological spaces occur as boundaries of groups?

In \cite{KK} Kapovich and Kleiner prove a classification theorem for boundaries of one-ended hyperbolic groups. They show that if the boundary is $1$-dimensional and the group does not split over a virtually cyclic subgroup then the boundary of the group is either a circle, a Sierpinski carpet, or a Menger curve.

\begin{problem}[K. Ruane]
Can the Kapovich-Kleiner result be extended to some natural family of $\CAT(0)$ groups?
\end{problem}

Kapovich and Kleiner's result relies heavily on JSJ results due to Bowditch \cite{Bow2}. Bowditch's results relate the existence of local cut points in the boundary to the existence of cut pairs, which is further related to two-ended splittings of the group. For $\CAT(0)$ groups Papasoglu and Swenson \cite{PS09} extend the connection between cut pairs and two-ended splittings, but leave the issue of local cut points completely unresolved. 

In this article we resolve this issue for groups acting {\it geometrically} (i.e. properly, cocompactly, and by isometries) on a $\CAT(0)$ space with isolated flats (see \cite{HK1}) and obtain the following result:

 \begin{theorem}[Main Theorem]
\label{theorem:MainTheorem}
Let $\Gamma$ be a group acting geometrically on a $\CAT(0)$ space $X$ with isolated flats. Assume $\bndry X$ is one-dimensional. If $\Gamma$ does not split over a virtually cyclic subgroup then one of the following holds:
\begin{enumerate}
\item $\bndry X$ is a circle
\item $\bndry X$ is a Sierpinski carpet
\item $\bndry X$ is a Menger curve.
\end{enumerate}
\end{theorem}

  A key tool used by Kapovich and Kleiner is the topological characterization of the Menger curve due to R.D. Anderson \cite{A58a, A58b}. Anderson's theorem states that a compact metric space $M$ is a Menger curve provided: M is $1$-dimensional, M is connected, M is locally connected, M has no local cut points, and no non-empty open subset of M is planar. We note that if the last condition is replaced with ``M is planar,'' then we have the topological characterization of the Sierpinski carpet (see Whyburn \cite{Whyburn}).

  Prior to Kapovich and Kleiner's theorem \cite{KK}, results of Bestvina and Mess \cite{BestvinaMess}, Swarup \cite{Swa}, and Bowditch \cite{Bow3} had shown that the boundary of a one-ended hyperbolic group $G$ is connected and locally connected. The planarity issue is easily dealt with using the dynamics of the action of the group on its boundary, leaving only the local cut point issue. However, Bowditch has shown \cite{Bow2} that if $\bndry G$ is not homeomorphic to a circle, then $\bndry G$ has a local cut point if and only if $G$ splits over a two-ended subgroup.

  We follow a similar outline to prove Theorem \ref{theorem:MainTheorem}. The question of which groups with isolated flats have locally connected boundary has been completely determined by Hruska and Ruane \cite{HR1}. So we will begin by assuming, for now, that $\bndry X$ is locally connected. In the isolated flats setting the planarity is again easily dealt with using an argument similar to that of Kapovich and Kleiner, leaving only the local cut point issue. So in order to complete the proof of Theorem \ref{theorem:MainTheorem}, the remaining difficulty is understanding the connection between local cut points in $\bndry X$ and splittings of $\Gamma$. 
  
  We prove the following splitting theorem which is independent of the dimension of $\bndry X$ and thus more general than is required for the proof of Theorem \ref{theorem:MainTheorem}: 
  
\begin{theorem}
\label{theorem:CutptTheorem}
Let $\Gamma$ be a group acting geometrically on a $\CAT(0)$ space $X$ with isolated flats. Suppose $\bndry X$ is locally connected and not homeomorphic to $S^1$. If $\Gamma$ does not split over a virtually cyclic subgroup, then $\bndry X$ has no local cut points.
\end{theorem}

\medskip
Techniques developed by the author for the proof of Theorem \ref{theorem:CutptTheorem} have already been used by Hruska and Ruane \cite{HR1} in the proof of their local connectedness theorem. In the special case when the boundary is one-dimensional Hruska and Ruane's \cite{HR1} theorem shows reduces to the statement that $\bndry X$ is locally connected if $\Gamma$ does not split over a two-ended subgroup (see Theorem \ref{theorem: HR}). So, we obtain a simplified version of Theorem \ref{theorem:CutptTheorem}, which is used in the proof of Theorem \ref{theorem:MainTheorem}.
  
\begin{corollary}
Let $\Gamma$ be a one-ended group acting geometrically on a $\CAT(0)$ space $X$ with isolated flats and assume that $\bndry X$ is 1-dimensional and not homeomorphic to $S^1$. If $\Gamma$ does not split over a virtually cyclic subgroup, then $\bndry X$ has no local cut points.
\end{corollary}  

Theorem \ref{theorem:CutptTheorem} fills a gap in the JSJ theory literature on $\CAT(0)$ groups and most of this paper is spent on the proof. We mention that the significance of this gap in the literature has also been observed by \Swiatkowski\, \cite{Swi16}. 

In \cite{Bow2} Bowditch studied local cut points in boundaries of hyperbolic groups and their relation to so called JSJ-splittings. As mentioned above, Bowditch showed that for a one-ended hyperbolic group $G$ which is not cocompact Fuchsian the existence of a local cut point is equivalent to a splitting of $G$ over a $2$-ended subgroup. Generalizations of Bowditch's results have been studied by Papasoglu-Swenson \cite{PS05} \cite{PS09}, and Groff \cite{G} for $\CAT(0)$ and relatively hyperbolic groups, respectively.

Groups with isolated flats have a natural relatively hyperbolic structure \cite{HK1}, and there is a strong relationship between $\bndry X$ and the Bowditch boundary $\relbndry$. Analogous to the limit set of a Kleinian group, the Bowditch boundary $\relbndry$ was introduced by Bowditch \cite{Bow1} and used to study splittings of hyperbolic groups. In general $\relbndry$ may have infintely many global cut points. In fact, the global cut point structure of $\relbndry$ is key to a general theory of splittings \cite{Bow4}.

Using cut pairs instead of local cut points Groff \cite{G} obtains a partial extension of Bowditch's JSJ tree construction \cite{Bow2} for relatively hyperbolic groups, and Guralnik \cite{Gur} observed that in the special case that the relative boundary $\relbndry$ has no global cut points, then many of Bowditch's results \cite{Bow2} about the valence of local cut points in the boundary of a hyperbolic group translate directly to the relatively hyperbolic setting. Their results were subsequently used by Groves and Manning \cite{GM} to show that if $\relbndry$ has no global cut points and all the peripheral subgroups are one-ended, then the existence of a local cut point in $\relbndry$ is equivalent to the existence of a splitting of $\Gamma$  relative to $\mathbb{P}$ over a non-parabolic $2$-ended subgroup. The relative boundary (or Bowditch boundary) $\relbndry$ is different from the $\CAT(0)$ boundary mentioned above. 

In \cite{Hau17a} the author investigates local cut points in $\relbndry$ and provides a splitting theorem for relatively hyperbolic groups without making any assumptions about global cut points. Namely, he shows that under some very modest conditions on the peripheral subgroups, the existence of a non-parabolic local cut point in $\relbndry$ implies that $\Gamma$ splits over a $2$-ended subgroup (see Theorem \ref{theorem:splitting in relative}). Because of the close relationship between $\bndry X$ and $\relbndry$, this splitting theorem will be used in Section \ref{sec:notinbndryF} to show that the existence of a local cut point of $\bndry X$ which is not in the boundary of a flat implies that $\Gamma$ splits over a $2$-ended subgroup.

We conclude the paper by discussing applications of Theorem \ref{theorem:MainTheorem}. In particular, in Section \ref{sec: Tripleexample} we discuss groups with Menger curve boundary and in Section \ref{sec: Sierpinski} we generalize the work of \Swiatkowski\, \cite{Swi16} to obtain the following result:
 
\begin{theorem}
\label{theorem: IFSwiatkowski}
Let $(W,S)$ be a Coxeter system such that $W$ has isolated flats. Assume that the nerve $L$ of the system is planar, distinct from simplex, and distinct from a triangulation of $S^1$. If the labeled nerve $L^{\bullet}$ of $(W,S)$ is distinct from a labeled wheel and inseparable, then $\bndry \Sigma$ is homeomorphic to the Sierpinski carpet. 
\end{theorem} 

Definitions of the terms used in Theorem \ref{theorem: IFSwiatkowski} can be found in Section \ref{sec: Sierpinski}. 

In \cite{DO01} Davis and Okun show that if $W$ is a Coxeter group whose nerve $L$ is planar, then $W$ acts properly on a $3$-manifold. Consequently, Theorem \ref{theorem: IFSwiatkowski} is in line with the following extension of a conjecture due to Kapovich and Kleiner \cite{KK}:

\begin{conjecture}
Let $\Gamma$ be a $\CAT(0)$ group with isolated flats and Sierpinski carpet boundary. Then $\Gamma$ acts properly on a contractible $3$-manifold.
\end{conjecture}

\subsection{Methods of Proof}
The strong connection between $\relbndry$ and the $\CAT(0)$ boundary $\bndry X$ is given by Hung Cong Tran \cite{Tran}. For spaces with isolated flats Tran's result implies that $\relbndry$ is the quotient space obtained from $\bndry X$ by identifying points which are in the boundary of the same flat. Using basic decomposition theory (see Section \ref{subsec: Decomposition}), we are able to show that if there exists a local cut point $\xi\in\bndry X$ that is not in the boundary of a flat, then it must push forward under this quotient map to a local cut point of $\relbndry$. This allows us to apply Theorem \ref{theorem:splitting in relative} mentioned above, and prove that the existence of a local cut point that is not in the boundary of a flat implies the existence of a $2$-ended splitting (see Proposition \ref{proposition:Cutptcase1}).

Assuming that our group $\Gamma$ does not split over a two ended group, we are left with the remaining question: Can a point which lies in the boundary of a flat be a local cut point? Much of this paper is spent answering that question in the negative when $\bndry X$ is locally connected. Let $X$ be a $\CAT(0)$ space which has isolated flats with respect to $\F$ and let $F\in\F$ be an $n$-dimensional flat in $X$, then 
it can be shown that $\stab_{\Gamma}(F)$ has a finite index subgroup $H$ isomorphic to $\bbz^n$. In Section \ref{sec:ProperCocompact} we show that $H$ acts properly and cocompactly on $\bndry X\setminus \bndry F$. This is done by means of a relation on $F\times (\bndry X\setminus\bndry F)$, which uses orthogonal rays to associate points in $F$ with points in the boundary. In Sections \ref{sec:addprop} and  \ref{sec:GeometricAction} we assume that $\bndry X$ is locally connected and show that we may put an $H$-equivariant metric on $\bndry X\setminus\bndry F$. Then in Section \ref{sec:inbndryF} we use the properties of this action to deduce that a point in the boundary of a flat cannot be a local cut point. This combined with Proposition \ref{proposition:Cutptcase1} allow us to obtain Theorem \ref{theorem:CutptTheorem}.

Once we have completed the proof of Theorem \ref{theorem:CutptTheorem} we are ready to prove Theorem \ref{theorem:MainTheorem}. This is accomplished in Section \ref{section: Main Theorem} with an argument inspired by Kapovich and Kleiner \cite{KK}. Using the dynamics of the action of $\Gamma$ on the boundary we show that if $\bndry X$ contains a non-planar graph $K$, then every open subset of must contain a homeomorphic copy of $K$. This will be enough to complete the proof.

\medskip

\subsection{Acknowledgments}
First and foremost, I would like to thank my advisor Chris Hruska for his guidance throughout this project. Second, I would like to thank the rest of UWM topology group Craig Guilbault, Ric Ancel, and Boris Okun for suggesting simplifications for some my  arguments. I would also, like to thank Kim Ruane, Jason Manning, and Kevin Schreve for helpful conversations.

\section{Preliminaries}
\label{sec:Preliminaries}

\subsection{The Bordification of a $\CAT(0)$ Space}
\label{subsec:CAT0Boundary}
Throughout this paper we will assume that $X$ is a proper $\CAT(0)$ metric space, unless otherwise stated. We refer the reader to \cite{BH1} for definitions and basic results about $\CAT(0)$ spaces.

 The {\it boundary} of $X$, denoted $\bndry X$, is the set of equivalence classes of geodesic rays. Where two rays $c_1,c_2\colon [0,\infty)\to X$ are equivalent if there exists a constant $D\geq 0$ such that $d\big(c_1(t),c_2(t)\big)\leq D$ for all $t\in[0,\infty)$. The {\it bordification} of $X$ is the set $\overline{X}=X\cup\bndry X$.

The bordification $\overline{X}$ comes equipped with a natural topology called the {\it cone topology}, where one considers rays based at some fixed point. A basis for the cone topology consists of open balls in $X$ together with ``neighborhoods of points at infinity.'' Given a geodesic ray $c$ and positive numbers $t>0$, $\epsilon >0$, define
\begin{gather*}
 V(c,t,\epsilon)=\bigset{\,x\in\overline{X} }{t<d\big(x,c(0)\big)\,\text{and} \, d\big(\pi_t(x),c(t)\big)<\epsilon}
\end{gather*}
Where $\pi_t$ is the orthogonal projection onto the closed ball $\overline{B}\big(c(0), t\big)$. For fixed $c$ and $\epsilon>0 $ the sets $V(c,t,\epsilon)$ form a neighborhood base at infinity about $c$. Intuitively, this means that two points in $\bndry X$ will be close if they are represented by rays which are $\epsilon$ close at for large values of $t$. We will denote by $V_{\bndry}(c,t,\epsilon)$ the set $V(c,t,\epsilon)\cap\bndry X$.

\subsection{The Bowditch Boundary and Splittings}
\label{subsec:The Bowditch Boundary and Splittings}
Let $\Gamma$ be a group and $\mathbb{P}$ a collection of infinite subgroups which is closed under conjugation, called {\it peripheral subgroups}.
\medskip

We say that $\Gamma$ is $hyperbolic$ $relative$ $to$ $\bbp$ if $\Gamma$ admits a proper isometric action on a proper $\delta$-hyperbolic space $Y$ such that:
\begin{enumerate}[(i)]
\item $\bbp$ is the set of all maximal parabolic subgroups
\item There exists a $\Gamma$-invariant system of disjoint open horoballs based at the parabolic points of $\Gamma$, such that if $\mathcal{B}$ is the union of these horoballs, then $\Gamma$ acts cocompactly on $Y\setminus\mathcal{B}$.
\end{enumerate}

The {\it Bowditch boundary} $\relbndry$ of $(\Gamma,\bbp)$ is defined to be the  boundary of the space $Y$. If $(\Gamma,\bbp)$ is relatively hyperbolic and acts geometrically on a $\CAT(0)$ space $X$ there is a close relationship between the $\relbndry$ and the visual boundary $\bndry X$.

\begin{theorem}[Tran]
\label{theorem:TranTheorem}$\relbndry$ is $\Gamma$-equivariantly homeomorphic to the quotient of $\bndry X$ obtained by identifying points which are in the boundary of the same flat.
\end{theorem}

A {\it splitting} of a group $\Gamma$ over a given class of subgroups is a finite graph of groups $\mathcal{G}$ of $\Gamma$, where each edge group belongs to the given class. 
A splitting is called {\it trivial} if there exists a vertex group equal to $\Gamma$. Assume that $\Gamma$ is hyperbolic relative to a collection $\bbp$. A {\it peripheral splitting} of $(\Gamma, \bbp)$ is a finite bipartite graph of groups representation of $\Gamma$, where $\bbp$ is the set of conjugacy classes of vertex groups of one color of the partition called peripheral vertices. Nonperipheral vertex groups will be referred to as {\it components}. This terminology stems from the correspondence between the cut point tree of $\relbndry$ and the peripheral splitting of $(\Gamma, \bbp)$, where elements of $\bbp$ correspond to cut point vertices and the components correspond to components of the boundary (i.e. equivalence classes of points not separated by cut points).

A peripheral splitting $\mathcal{G}$ is a refinement of another peripheral splitting $\mathcal{G}'$ if $\mathcal{G}'$ can be obtained from $\mathcal{G}$ via a finite sequence of foldings that preserve the vertex coloring. In \cite{Bow4} Bowditch proved the following accessibility result:

\begin{theorem}
Suppose that $(\Gamma,\bbp)$ is relatively hyperbolic and that $\relbndry$ is connected. Then $(\Gamma,\bbp)$ admits a (possibly trivial) peripheral splitting which is maximal in the sense that it is not a refinement of any other peripheral splitting.
\end{theorem}

\subsection{Isolated Flats}
\label{subsec:IsolatedFlats}

Here we introduce basic definitions and pertinent results regarding spaces with isolated flats. We refer the reader to \cite{HK1} for a more detailed account. Let $X$ be a $\CAT(0)$ space with $\Gamma$ acting geometrically on $X$. A {\it k-flat} in $X$ is an isometrically embedded copy of Euclidean space, $\mathbb{E}^k$. A $1$-flat will also be referred to as a {\it line} and a $2$-flat may be referred to as a {\it flat plane}.

The space $X$ is said to have {\it isolated flats} if there is a $\Gamma$-invariant collection of  flats, $\mathcal{F}$, of dimension 2 or greater and such that the following hold:
\medskip

\begin{enumerate}[(i)]

\item (capturing condition) There exists a constant $D<\infty$ such that each flat in $X$ lies in the $D$-tubular neighborhood of some $F\in\mathcal{F}$

\item (isolating condition) For every $\rho<\infty$ there exists $\kappa(\rho)<\infty$ such that for any two distinct $F,F'\in \mathcal(F)$ we have $\diam\big(\mathcal{N}_{\rho}(F)\cap\mathcal{N}_{\rho}(F')\big)<\kappa(\rho)$

\end{enumerate}

\medskip
Hruska and Kleiner have shown in \cite{HK1} that if $\Gamma$ is a group acting geometrically on a $\CAT(0)$ space with isolated flats, then $\Gamma$ is hyperbolic relative to a collection of virtually abelian subgroups of rank at least 2 (Theorem 1.2.1 of \cite{HK1}). Hruska and Kleiner have also shown that for isolated flats $\bndry X$ is an invariant of the group $\Gamma$ up to quasi-isometry (Theorem 1.2.2 of \cite{HK1}). The following recent result concerning groups with isolated flats is due to Hruska and Ruane \cite{HR1}, and is particularly relevant to this project:

\begin{theorem}[Hruska-Ruane]
\label{theorem: HR}
Let $\Gamma$ be a one-ended group acting geometrically on a $\CAT(0)$ space with isolated flats.
Let $\mathcal{G}$ be the maximal peripheral splitting of $\,\Gamma$.
Then each vertex group of $\mathcal{G}$ acts geometrically
on a $\CAT(0)$ space with locally connected boundary.

Furthermore $\bndry X$ is locally connected if and only if the following
condition holds:
Each edge group of $\mathcal{G}$ has finite index in the adjacent
peripheral vertex group.

\end{theorem}

In the case where $\bndry X$ is $1$-dimensional we have the following corollary:

\begin{corollary}
\label{corollary: HRcorollary}
Assume $\Gamma$ is acting geometrically on a $\CAT(0)$ space $X$ with isolated flats, and assume $\bndry X$ is $1$-dimensional. Then $\bndry X$ is locally connected if and only if $\Gamma$ does not have a peripheral splitting over a $2$-ended subgroup.
\end{corollary}

\begin{remark} As ``no splitting over a two-ended subgroup'' is a hypothesis in both Theorem \ref{theorem:MainTheorem} and Theorem \ref{theorem:CutptTheorem}, we may assume that $\bndry X$ is locally connected when required. Also, notice that for the proof of Theorem \ref{theorem:MainTheorem} we are concerned with 1-dimensional boundaries, so in that case the dimension of the flats we are interested in is 2. However, for many of the result we will not need to make any assumption about the dimension of flats.
\end{remark}

\subsection{Local Cut Points}
\label{subsec:Cut Points and Cut Pairs In Metric Spaces}

Recall that a {\it continuum} is a non-empty, connected, compact, metric space, and let $M$ be such a space. A {\it cut point} of $M$ is a point $x\in M$ such that $M\setminus \{x\}$ is disconnected. A point $x\in M$ is a {\it local cut point} if $x$ is a cut point or $M\setminus\{x\}$ has more than one end. A detailed discussion of ends of spaces can be found in Section 3 of \cite{Gui13}. In this paper we are often interested in whether a given point is a local cut point or not. Thus we remark that saying a point $x\in M$ is a local cut point is equivalent to saying that there exists a neighborhood $U$ of $x$ such that for every neighborhood $V$ of $x$ with $V\subset U$, there exist points $z,y\in V\setminus\{x\}$ which cannot be connected inside $U\setminus\{x\}$, i.e. $z$ and $y$ are not contained in the same connected subset of $U\setminus\{x\}$. In Section \ref{sec:inbndryF} we will be interested in showing that a point cannot be a local cut point, so it is worth noting the negation of the above. In other words, to check that $x$ is not a local cut point it suffices to show that given a neighborhood $U$ of $x$ there exists a neighborhood $V\ni x$ with $V\subset U$ and $V\setminus \{x\}$ connected.

In his study of JSJ splittings of hyperbolic groups Bowditch investigated the local cut point structure of the boundary. In that setting Bowditch shows that the existence of a local cut point implies that group splits over a 2-ended subgroup. In \cite{Hau17a} the author studies local cut points in the relative boundary (or Bowditch Boundary) $\relbndry$ and has generalized Bowditch's result to show:

 \begin{theorem}
\label{theorem:splitting in relative}
Let $(\Gamma,\bbp)$ be a relatively hyperbolic group and suppose each $P\in\bbp$
is finitely presented, one- or two-ended, and contains no infinite torsion
subgroup. Assume that $\relbndry$ is connected and not homeomorphic to a circle. If $\relbndry$ contains a non-parabolic local cut point, then $G$ splits over a $2$-ended subgroup.
\end{theorem}

The majority of this paper is concerned with determining the existence or non-existence of local cut points $\bndry X$. Theorem \ref{theorem:splitting in relative} will be used in Section \ref{sec:notinbndryF} to show that the existence of a local cut point which is not in the boundary of a flat implies the existence of a splitting over 2-ended subgroup.

\subsection{Limit Sets}
\label{subsection:Limit Sets}

We will need a few basic results about limit sets sporadically through this paper, consequently, we conclude the preliminary section with a terse discussion of limit sets. In this section $X$ will be a $\CAT(0)$ space and $\Gamma$ some group of isometries of $X$.

Recall, that a for a sequence $(\gamma_n)\subset G$ we write $\gamma_n\rightarrow \xi\in\bndry X$ if $\gamma_n x\rightarrow\xi$ for some $x\in X$. It is clear that if $\gamma_n x\rightarrow\xi$ for some $x$, then $\gamma_n x'\rightarrow\xi$ for any $x'\in X$. The {\it limit set}, $\Lambda(\Gamma)$, of $\Gamma$ is the subset of $\bndry X$ consisting of all such limits. The set $\Lambda(\Gamma)$ is a closed and $\Gamma$-invariant. Given that the action of $\Gamma$ is geometric we have the following:

\begin{lemma}
$\Lambda(\Gamma)=\bndry X$
\end{lemma}

We leave the proof of this result as an exercise.

A subset $M$ of $\Lambda (\Gamma)$ is said to be {\it minimal} if $M$ is closed, non-empty, $\Gamma$-invariant, and does not properly contain a closed $\Gamma$-invariant subset. A useful fact about minimal sets is that $M\subset\Lambda(\Gamma)$ is minimal if and only if $\orb_{\Gamma}(m)$ is dense in $M$ for every $m\in M$. 
The action of $\Gamma$ on $\Lambda(\Gamma)$ is called {\it minimal} if $\Lambda(\Gamma)$ is minimal.

\section{A Proper and Cocompact Action on $\bndry X\backslash \bndry F$}
\label{sec:ProperCocompact}

Let $X$ be a $\CAT(0)$ space with isolated flats and let $F\in\F$ be a flat in $X$. Set $Y=\bndry X\setminus\bndry F$.  In this section we a follow a strategy similar to that of Bowditch in Lemma 6.3 of \cite{Bow1} to show that $\stab_{\Gamma}(F)$ acts properly and cocompactly on $Y$. The key observation made by Bowditch is as follows:

\begin{lemma}
\label{lemma:diagonalaction}
 Let $G$ be a group acting on topological spaces $A$ and $B$. Define the action of $G$ on $A\times B$ to be the diagonal action and let $\R\subset A\times B$. If $\R$ is $G$-invariant and the projections $pr_A$ and $pr_B$ from $\R$ onto the factors are both proper and surjective, then the following are equivalent:
  \begin{enumerate}[(a)]
  \item $G$ acts properly and cocompactly on $A$
  \item $G$ acts properly and cocompactly on $\R$
  \item $G$ acts properly and cocompactly on $B$
  \end{enumerate}

\end{lemma}

\medskip
 Set $G=\stab_{\Gamma}(F)$. Define $\perp(F)$ be the set of all geodesic rays orthogonal to $F$. Recall that a geodsic ray $r\colon [0,\infty)\to X$ is orthogonal to a convex set $C\subset X$ if for every $t>0$ and for any $y\in C$ the Alexandrov angle, $\angle_{r(0)}\big(r(t),y\big)$, is greater than or equal to $\pi/2$. It is well known that $G$ acts cocompactly on $F$ (see \cite{HK1} Lemma 3.1.2). Let $A>0$ be the diameter of the fundamental domain of this action. We define $\R=\big\{ (x,q)\big| \spc q\in\perp(F) \spc \text{with}\spc d\big(x,q(0)\big)\leq A\big\}$. Unless otherwise stated, we will assume that our base point $x_0$ is in the flat $F$.

We want that $\R$ satisfies the hypotheses of Lemma \ref{lemma:diagonalaction}, with the roles of $A$ and $B$ played by $F$ and $Y=\bndry X\setminus F$. We begin with the following observation:

\medskip
\medskip
\begin{lemma}
$\R$ is $G$-invariant.
\end{lemma}

\begin{proof}
$G\leq \Gamma$ acts on $X$ by isometries, so if $(x,q)\in\R$ and $h\in G$ then $d\big(h.x,h.q(0)\big)<A$. If $\pi_F$ is the orthogonal projection onto $F$, then by $d\big(h.q(t),h.q(0)\big)=d\big(q(t),q(0)\big)$ and the uniqueness of the projection point $\pi_F$ (see \cite{BH1} Proposition II.2.4) we must have that $\pi_F\big(h.q(t)\big)=h.\pi_F\big(q(t)\big)$ for every $t\in [0,\infty)$.
\end{proof}

\medskip
 To continue our study of $\R$, we require the following useful lemma, which allows one to construct a new orthogonal ray from a sequence of orthogonal rays with convergent base points. The proof relies on a standard diagonal argument and will not be presented here; however, it is not dissimilar to the proof presented in Lemma 5.31 of \cite{BH1}.

\begin{lemma}
\label{lemma:diagonallemma}
If $(Y,\rho)$ is a separable metric space, $(X,d)$ is proper metric space, $y_0\in Y$, and $K$ a compact subset of $X$, then any sequence of isometric embeddings, $c_n\colon Y\to X$, with $c_n(y_0)\in K$ has a subsequence which converges point-wise to an isometric embedding $c\colon Y\to X$.
\end{lemma}

Next, we check that $\R$ projects surjectively onto the factors $Y$ and $F$.

\begin{lemma}
Let $\xi\in Y$. If $r$ is a ray representing $\xi$, then there exists a geodesic ray $q\in \perp(F)$ asymptotic to $r$.
\end{lemma}

\begin{proof}Let $x_n=r(n)$ and $y_n=pr_F(x_n)$ for all $n\in \bbn$. Our first claim is that the sequence $(y_n)$ is bounded as $n\rightarrow \infty$. Assume not, then $y_n\rightarrow \eta\in\bndry F$. By Corollary 7 of \cite{HK2} there exists some constant $M>0$ such that $d\big(x_0, [x_n,y_n]\big)<M$ for all $n$. Then by the triangle inequality and the definition of $y_n$ as the orthogonal projection we have that $d(x_0,y_n)\leq 2M$. Then $(y_n)$ converges in $\overline{B}(x_0, 2M)$ and we may apply Lemma \ref{lemma:diagonallemma} to construct the orthogonal ray $q$.
\end{proof}

\begin{corollary}
\label{corollary: R surjectivity}
The projections $pr_Y(\R)$ and $pr_F(\R)$ are surjective.
\end{corollary}

\begin{proof} The surjectivity of $pr_Y$ is immediate. For $pr_F$ we need only that each point in the flat is within a bounded distance of an element of $\perp(F)$. Let $A>0$ be the constant used in the definition of $\R$ above. We know from the previous lemma that there exists some $q\in\perp (F)$. The result follows as $\orb_G(q)\subset\perp(F)$ and $\orb_G(q)\cap F$ is $A$-dense.
\end{proof}

\medskip
\medskip
In order to check the properness of the projections, we need to know that as a sequence $(r_n)$ of orthogonal rays moves the corresponding sequence of asymptotic rays based at $x_0$ travel within a bounded distance of the points $\big(r_n(0)\big)$. We provide a quasiconvexity result below, which is a corollary of the following theorem presented by Hruska and Ruane in 4.14 of Theorem \cite{HR1}.

\begin{theorem}
\label{theorem:HRqc}
Let $X$ be a $\CAT(0)$ space with isolated flats with respect to $\F$. There exists a constant $L>0$ such that the following hold:
\begin{enumerate}
\item Given two flats $F_1,F_2\in\F$ with $c$ the shortest length geodesic from $F_1$ to $F_2$, we have that $F_1\cup F_2\cup c$ is $L$-quasiconvex in $X$.

\item Given a point $p$ and a flat $F\in\F$, with $c$ the shortest path from $F$ to $p$, then $F\cup c$ is $L$-quasiconvex in $X$.

\end{enumerate}

\end{theorem}

\begin{lemma}
\label{corollary:quasiconvex}
Let $F\in\mathcal{F}$ and $q\in\perp(F)$ then there exists a constant $L$ such that $q\cup F$ is $L$-quasiconvex in $X$.
\end{lemma}

\begin{proof} The proof follows from the Theorem \ref{theorem:HRqc} by passing to the limit as $n\rightarrow\infty$ of the geodesic segments $\big[q(n),q(0)\big]$.
\end{proof}

\medskip
To prove properness we will also need to know that convergence of orthogonal rays in the space $X$ corresponds to convergence of points in $Y$.

\medskip
\begin{lemma}
\label{lemma: convergence at infinity}
 If $(r_n$) is a sequence of rays in $\perp(F)$ which converge pointwise to an element of $\perp(F)$, then the corresponding points at infinity converge in the topology on $\bndry X$.
\end{lemma}

\begin{proof} Let $x_0$ be the base point for the cone topology on $\bndry X$ and let $r\in\perp(F)$ be limit ray. For each ray $r_n$ there exists a an asymptotic ray $c_n$ based at $x_0$ (see \cite{BH1} Chapter II.8 Proposition 8.2). Define $D=d\big(x_0,r(0)\big)$, then $d\big(r_n(t),c_n(t)\big)\leq D$ for every $t\in\bbr$. Thus we may apply Lemma \ref{lemma:diagonallemma} with $K=\overline{B}(r(0),D)$ to find the limiting based ray $c$. Then $c$ is asymptotic to $r$. The claim is that $c_n(\infty)\rightarrow c(\infty)$ in the cone topology. Fix $\epsilon>0$ and let $s>0$. Then $U(c,s,\epsilon)=\bigset{c'\in\bndry_{x_0}X}{d\big(c'(s),c(s)\big)<\epsilon}$ is a basic neighborhood of $c$. As $c_n\rightarrow c$ pointwise we have that there exists $N\in \bbn$ $d\big(c_m(s),c(s)\big)<\epsilon$ for every $m>N$. Thus we have the claim.
\end{proof}

\medskip
We now wish to prove that the projections $pr_Y$ and $pr_F$ are proper. In order to do so we will need several lemmas concerning the relationship between base points of orthogonal rays and the based rays which represent them in the boundary (See Figure 1 for an intuitive picture).

\begin{figure}[h!]
\includegraphics[scale=1.3]{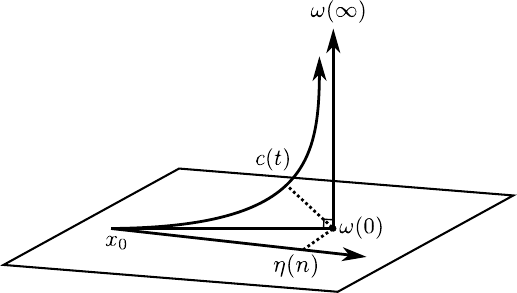}
\captionof{figure}{This figure illustrates the relationship between an orthogonal ray $\omega$, its based representative $c$, and a neighborhood $U(\eta, n,\epsilon)$.}
\label{fig: proj to flat}
\end{figure}

\begin{lemma}
\label{lemma:bound1}
There is a constant $M=M(L)>0$ such that for any $\omega\in\perp(F)$, $d\big(\omega(0),c(t)\big)<M$ where $c$ is the ray based at $x_0$ representing $\omega(\infty)$ and $t=d\big(x_0,\omega(0)\big)$.

\end{lemma}

\begin{proof} Let $\beta \colon [0,t]\to F$ be the geodesic with $\beta(0)=x_0$ and $\beta(t)=\omega(0)$. By Corollary \ref{corollary:quasiconvex} there exists a constant $L$ such that $c$ is contained within the $L$-tubular neighborhood of $\omega\cup F$. This implies that there exists an $s\in[0,\infty)$, $x\in im(\omega)$, and $y\in F$ with $d\big(c(s),x\big)\leq L$ and $d\big(c(s),y\big)\leq L$. By orthogonal projection we know that $d\big(x,\omega(0)\big)\leq 2L$, which implies that $d\big(c(s),\omega(0)\big)\leq 3L$. Because, $\beta$ and $c$ are geodesics the triangle inequality gives us that $t\in \big[s-3L,s+3L\big]$. Thus setting $M=6L$ we are done.
\end{proof}

\begin{lemma}
\label{lemma:bound2}
Suppose $\epsilon>0$ and let $M=M(L)$ be the constant found in Lemma \ref{lemma:bound1}. There is a constant $\delta=\delta(M,\epsilon)$ such that for any $n>0$ and $\eta\in\bndry F$ the following holds: if $\omega\in\perp(F)$ with $\omega(0)\in U(\eta,n,\epsilon)$ and $c$ is the ray based at $x_0$ representing $\omega(\infty)$, then $c(n)\in U(\eta,n,\delta)$.

\end{lemma}

\begin{proof} Let $\beta\colon [0,a]\to F$ be the geodesic with $\beta(0)=x_0$ and $\beta(a)=\omega(0)$. If $M$ is the constant from Lemma \ref{lemma:bound1}, then we know that $d\big(\beta(a),c(a)\big)\leq M$. So, if $a=n$ we are done.

Assume that $n<a$. Then by convexity $d\big(c(n),\beta(n)\big)\leq M$ and $d\big(\beta(n),\eta(n)\big)\leq\epsilon$, which implies that $d\big(c(n),\eta(n)\big)\leq M +\epsilon$.

If $a<n$, then $\beta(a)=\beta(n)$. By hypothesis $\beta(a)\in U(\eta, n,\epsilon,)$, which implies that $d\big(\beta(a),\eta(n)\big)<\epsilon$. So, $d\big(c(a),\eta(n)\big)\leq M +\epsilon$, but $c$ and $\eta$ are geodesics so we have that $d\big(c(n),\eta(n)\big)\leq 2(M +\epsilon)$. Set $\delta=2(M+\epsilon)$. 
\end{proof}

\begin{lemma}
\label{lemma:bound3}
Let $W\subset Y$ be compact. The set $C$ of all points $x\in F$ with $d\big(x,w(0)\big)\leq A$ for some $w\in W$ is bounded.
\end{lemma}

\begin{proof} Assume not, then there exists a sequence $(c_n)$ in $\perp(F)$ with $c_n(0)\in C$ for every $n\in\bbn$ such that $c_n(0)\rightarrow \eta$ as $n\rightarrow\infty$ for some $\eta\in\bndry F$. For every $n$ let $r_n$ be the ray based at the base point $x_0\in F$ and asymptotic to $c_n$.

Recall that for any $D$ the sets $U(\eta, t,D)$ form a neighborhood base at $\eta$. Fix $\epsilon >0$. Then $c_n(0)\rightarrow \eta$ implies that for any $t\in[0,\infty)$ we have $c_n(0)$ lies in $U(\eta, t,\epsilon)$ for all but finitely many $n$. Lemma \ref{lemma:bound2} then gives us that for any $t\in[0,\infty)$ we have $r_n(0)$ lies in $U(\eta, t,\delta)$ for all but finitely many $n$, which by Lemma \ref{lemma: convergence at infinity} implies that $r_n(\infty)\rightarrow \eta$, a contradiction.
\end{proof}

\medskip
\begin{lemma}
\label{lemma:proper projection}
The projections $pr_Y(\R)$ and $pr_F(\R)$ are proper.
\end{lemma}

\begin{proof} Let $K\subset F$ be compact. We want that $pr_F^{-1}(K)$ is compact. Let $(r_n$) be a sequence of rays in $\perp(F)$ with base points in $K$. As $K$ is compact using Lemma \ref{lemma:diagonallemma} we know that the sequence $(r_n)$ has a subsequence which converges to a ray $r$ based in $K$. Set $x_n=r_n(1)$, $p_n=r_n(0)$, and let $y$ be some fixed point in $F$. We know that $\angle_{p_n}(y,x_n)\geq\pi/2$ and that the function $(p,x,y)\mapsto\angle_p(x,y)$ is upper semi-continuous for all $p,x,y\in X$ (see \cite{BH1} Proposition II.3.3(1)), thus $r$ must be a ray orthogonal to $F$. By Lemma \ref{lemma: convergence at infinity} we have that the sequence of points at infinity converges, which implies that $pr_F^{-1}(K)$ is compact.

Now, let $W$ be a compact subset of $Y$ and $C$ the set of all points $x\in F$ with $d\big(x,w(0)\big)\leq A$ for some $w\in W$. We need that $C$ is compact. By Lemma \ref{lemma:bound3} the set $C$ is bounded. We only need that $C$ is closed.

Assume that $c$ is a limit point of $C$. Then there exists a sequence $(c_i)_{i=0}^{\infty}$ of points in $C$ which converge to $c$, and there exists a sequence of rays $(w_i)_{i=0}^{\infty}$ in $\perp(F)$ with $d\big(c_i,w_i(0)\big)\leq A$ and $w_i(\infty)\in W$ for every $i$. The sequence $\big(w_i(0)\big)_{i=0}^{\infty}$ converges in $\overline{\N}_A(C)$, so by a diagonal argument we have that $(w_i)$ converges to a ray $w\in\perp(F)$. Lemma \ref{lemma: convergence at infinity} and compactness of $W$ give that $w(\infty)\in W$. It is now easy to see that $d\big(c,w(0)\big)\leq A$.\end{proof}

\medskip
Combining the previous results we may apply Lemma \ref{lemma:diagonalaction} to conclude:

\begin{theorem}
\label{theorem: prprccpt}
 Let $G =\stab_{\Gamma}(F)$. Then $G$ acts properly and cocompactly on $\bndry X\backslash\bndry F$.
 \end{theorem}

 As mentioned above in Lemma 3.1.2 of \cite{HK1} Hruska and Kleiner showed that $G=\stab_{\Gamma}(F)$ acts cocompactly on $F$. The Beiberbach theorem then gives that $G$ contains a subgroup of finite index $H$ isomorphic to $\bbz^n$, where $n$ is the rank of the flat $F$.

We then obtain the following corollary:

\begin{corollary}
\label{corollary:Hpropercocmpct}
The subgroup $H$ acts properly and cocompactly on $\bndry X\backslash\bndry F$.
\end{corollary}

\section{Additional Properties of $\bndry X\setminus \bndry F$}
\label{sec:addprop}

Throughout this section we will assume that $\bndry X$ is locally connected. Let $\C$ be the collection of connected components of $Y=\bndry X\backslash \bndry F$. The goal of this section is to show that $\C$ has finitely many orbits and that the stabilizer of each $C\in\C$ has a finite index subgroup isomorphic to $\bbz^n$, where $n$ is the dimension of the flat. This fact will play a crucial role in Sections \ref{sec:GeometricAction} and \ref{sec:inbndryF}.

Let $Z$ be a closed convex subset of a metric space $M$, and let $G$ be any subgroup of $\isom(M)$. We say $Z$ is $G$-{\it periodic} if $\stab_{G}(Z)$ acts cocompactly on $Z$. As in the previous section let $H\leq \stab_{\Gamma}(F)$ be a finite index subgroup isomorphic to $\bbz^n$, where $n$ is the dimension of $F$. We begin with two results concerning the $H$-periodicity of elements of $\C$ that will be needed to prove the main result of this section. 
\subsection{H-periodicity}
\label{subsec: Hperiodic}

\medskip
\medskip

\begin{lemma}
\label{lemma: locally finite}
The collection $\C$ is locally finite, i.e only finitely many $C\in\C$ intersect any compact set $K\subset Y.$
\end{lemma}

\begin{proof} This simply follows from the local connectedness of $Y$. Assume that $\C$ is not locally finite. Then there exists $K\subset Y$ such that $K$ meets infinitely many elements of $\C$. We may then find a sequence $(x_C)_{C\in\C}$ of points from distinct elements of $\C$ which meet $K$. This sequence must converge to a point $x$ in $K$. Thus any neighborhood of $x$ meets infinitely many members of $\C$. $Y$ is an open subset of a locally connected space and thus must be locally connected, a contradiction. \end{proof}

\medskip

\begin{lemma}
\label{lemma:Hperiodic}
Let $\C$ be as above. Then we have the following:
\begin{enumerate}[(i)]
\item The elements of $\mathcal{C}$ lie in only finitely many $H$-orbits.
\item Each $C\in \mathcal{C}$ is $H$-periodic, i.e $\stab_H(C)$ acts cocompactly on $C$.
\end{enumerate}

\end{lemma}

\medskip
\begin{proof} From Lemma \ref{lemma: locally finite} we know that $\C$ is locally finite, and we saw in Corollary \ref{corollary:Hpropercocmpct} that $H$ acts properly and cocompactly on on $Y$. We may now follow word for word the proof of Lemma 3.1.2 of \cite{HK1}.\end{proof}

\subsection{Full Rank Components}
\label{subsec:FullRank}

\medskip
\medskip
\begin{lemma}
\label{lemma:convcorrespondence}
 Assume we have a sequence of rays in $(r_i)\subset\perp(F)$ with base points, $r_i(0)$, converging to a point $\xi$ in $\bndry F$, then the sequence $\big(r_i(\infty)\big)$ converges to $\xi$ in $\bndry F$.
\end{lemma}

\begin{proof} By Corollary \ref{corollary:quasiconvex} each point $r_i(\infty)$ is represented by based rays $c_i$ that stay within the $L$-neighborhood of  $F\cup$im$(r_i)$. Thus if $r_i(0)\rightarrow\xi\in\bndry F$, then for any $n\in \bbn$ all but finitely many members of the sequence $\big(r_i(0)\big)$ are inside $U(\xi, n,\epsilon)$, which implies that for any $n$ all but finitely many $c_i$ lie in $U(\xi, n,\epsilon + L)$.\end{proof}

\begin{corollary}
\label{corollary:limitsetcorollary}
Every point in $\bndry F$ is a limit point of points in $\bndry X\backslash \bndry F$.
\end{corollary}

\begin{proof} Let $\xi$ be in the boundary of a $F$ and $c\colon [0,\infty)\to X$ a based ray representing $\xi$. Then by Corollary \ref{corollary: R surjectivity} for each $c(n)$ there is an $r_n\in\perp (F)$ such that $d\big(c(n), r_n(0)\big)<A$. So, the sequence $r_n(0)$ converges to $\xi$ and we may apply the preceding lemma. \end{proof}

\medskip
The proof of the following lemma essentially amounts to checking Bestvina's nullity condition \cite{Bes96} for the action of $H$ on $\bndry X\setminus\bndry F$.

\begin{lemma}
\label{lemma:components are asymptotic}
If $C\in\C$ then elements of $\orb_H(C)$ are asymptotic in the sense that two components meet in $\Lambda\big(\stab_H(C)\big)\subset \bndry F$.
\end{lemma}

\begin{proof} Let $C'\in \orb_H(C)$ and let $c'_n$ a sequence of points in $C'$ converging to a point of $\bndry F$. We show that there is a sequence of points in $C$ which converge to the same point of $\bndry F$.

Each $c'_n$ is a translate of some point some point $c_n$ in $C$. Notice that because $H$ is abelian $\stab_H(C')=\stab_H(C)$, and that by Lemma \ref{lemma:Hperiodic} there exists compact sets $K'$ and $K$ whose $\stab_h(C)$-translates cover $C'$ and $C$, respectively. So there exists a sequence of group elements $(h_n)$ in $\stab_H(C)$ such that $c'_n$ is contained in $h_nK'$ and $c_n\in h_nK$ for every $n$. 

Now, as in Section \ref{sec:ProperCocompact} consider the projection of $K$ and $K'$ to the the flat $F$ and choose two points $k'\in K'$ and $k\in K$. For every $n$ the point $h_nk'$ is within a bounded distance of the base of an orthogonal representative for $c'_n$. Similarly, each $h_nk$ is within a bounded distance of an orthogonal representative of $c_n$; moreover, the distance $d(k,k')$ is bounded. So, $(h_nk)$ and $(h_nk')$ converge to the same point $\xi$ in $\bndry F$, which implies that the bases of the orthogonal representatives of the $(c_n)$ and $(c'_n)$ also converge to $\xi$. We may apply Lemma \ref{lemma:convcorrespondence} to complete the proof. \end{proof}

\medskip
\begin{proposition}
\label{proposition:fullrank}
Let $C\in\C$, then $C$ is connected with stabilizer isomorphic to $\bbz^n$.
\end{proposition}

\begin{proof}By way of contradiction suppose $C\in\C$ and assume that $H'=\stab_H(C)\cong \bbz^k$ for some $k<n$. As $k<n$ we may find an $h\in H\setminus H'$ and an axis $\ell\colon\bbr\to F$ passing through $x_0\in F$ for $h$ with $\ell(+\infty)$ and $\ell(-\infty)$ not in the limit set $\Lambda(H')$. Let $\xi$ be the point of $\bndry F$ represented by the ray $\ell|[0,\infty)$. As $\Lambda(H')$ is a closed subsphere of $\bndry F$, we may find a $U =U(\xi,n,\epsilon)$ neighborhood of $\xi$  in $\bndry X\setminus\Lambda(H')$.

 Fix $\eta \in C$ and let $r$ be an orthogonal representative of $\eta$. Then $h^n(r)$ is an orthogonal ray for every $n$ and the sequence $h^n\big(r(0)\big)$ converges to $\xi$, which by  Lemma \ref{lemma:convcorrespondence} implies $h^n(\eta)\rightarrow\xi$. As each $h^n(\eta)$ lies in a different element of $\orb_{\langle h\rangle}(C)$ we have that infinitely members of $\orb_{\langle h\rangle}(C)$ intersect $U$. As $\langle h\rangle$ stabilizes $\bndry F$ and the orbit under $H'$ of points in $C$ converges to points in $\Lambda(H')$, Lemma \ref{lemma:components are asymptotic} implies that no element of $\orb_{\langle h\rangle}(C)$ is contained in $U$. Thus $\bndry X$ is not locally connected, a contradiction.\end{proof}

\medskip
Combining this result with the $H$-periodicity result Lemma \ref{lemma:Hperiodic} we obtain:

\begin{corollary}
\label{corollary:finitelymany}
There are only finitely many components of $\bndry X\setminus \bndry F$.
\end{corollary}


\section{An Equivariant Metric on $\bndry X\backslash \bndry F$}
\label{sec:GeometricAction}

In this section we assume that $Y=\bndry X\setminus \bndry F$ is locally connected. Let $H$ be the maximal free abelian subgroup of $\stab_{\Gamma}(F)$. We will put an $H$-equivariant metric on $Y$. First, we remind the reader of a standard result about covering spaces that will be used several times throughout this section:

\begin{lemma}
Let $G$ be a torsion free group acting properly on a locally compact Hausdorff space $X$, then $X$ together with the quotient map $q\colon X\to X/G$ form a normal covering space of $X/G$.
\end{lemma}

\medskip
We begin with a review of how one defines the {\it pull-back length metric} of a length space. We refer the reader to \cite{P} for a more detailed account. Recall that a length metric is one where the distance between two points is given by taking the infimum of the lengths of all rectifiable curves between $x$ and $y$. Suppose that $X$ is a length space and $\XT$ is topological space, and $p\colon\XT \to X$ is a surjective local homeomorphism. Define a pseudometric on $\XT$ by:
\[
\tilde{d}(\xt,\yt)=\inf\bigset{L(p\circ\tilde{\gamma})}{ \tilde{\gamma}\colon [0,1]\to\XT\text{  a curve from}\spc \tilde{x}\spc \text{to}\spc \tilde{y}}
\]

Where $L(p\circ\tilde{\gamma})$ is the length of the path $p\circ\tilde{\gamma}$. If $\XT$ is Hausdorff then $\tilde{d}$ is a length metric (see \cite{P} Proposition 3.4.7). Also, it is easy to show that:

\begin{lemma}
If $X$ is obtained as the quotient of a free and proper action by a group $G$ then the metric $\tilde{d}$ is $G$-equivariant.
\end{lemma}

\begin{proof} Let $P(\xt,\yt)$ be the set of all paths between $\xt$ and $\yt\in \XT$ and $Q(x',y')=\bigset{p\circ\sigma}{\sigma\in P(x',y')}$. To prove that $\tilde{d}(\xt,\yt)=\tilde{d}(g\xt,g\xt)$ it suffices to show that $Q(\xt,\yt)=Q(g\xt,g\yt)$. But this is clear, as $X$ is obtained as the quotient of the group action, i.e. if $\gamma$ is a path in $\XT$, then  $\gamma$ and $g\gamma$ are identified.\end{proof}

\medskip
\medskip
Let $G$ be a torsion free group, acting, properly and cocompactly on a connected component $C$ of $\bndry X\backslash\bndry F$, set $Q=C/G$, and define $q\colon C\to C/G$ to be the associated quotient map. In order to apply the above construction to our setting we need that $Q$ is a length space. I would like to thank Ric Ancel for pointing out the following theorem due to R.H. Bing (see \cite{Bing1}), which we will use to show that $Q$ is a length space:

\begin{theorem}[Bing]
\label{theorem:Bing's Theorem}
Every Peano continuum admits a convex metric.
\end{theorem}

Recall that that a {\it Peano continuum} is a compact, connected, locally connected metrizable space. The notion of convexity used by Bing is that of Menger convexity. For proper metric spaces Menger convexity is known to be equivalent to being geodesic \cite{P}.  Recall that a {\it geodesic}, between two points $x$ and $y$ in a metric space $X$ is an isometric embedding of an interval $\gamma\colon [0,D]\to X$ such that $\gamma(0)=x$, $\gamma(D)=y$, and $D=d_X(x,y)$. By a {\it geodesic metric space} we mean that there is a geodesic joining any two points of the space. Compact metric spaces are proper, so we may replace the word ``convex'' with ``geodesic'' in Bing's result. Note that by default a geodesic metric space a length space. Therefore, we need only show that $Q$ is a Peano continuum to obtain that $Q$ is a length space.

To show that $Q$ is metrizable we use Urysohn's metrization theorem:

\begin{theorem}
Let $X$ be a $T_1$ space. If $X$ is regular and second countable, then $X$ is separable and metrizable.
\end{theorem}

This theorem and all general topology results used in this section can be found in \cite{Wil1}.

\medskip
\begin{lemma}
$Q$ is second countable.
\end{lemma}

\begin{proof} First note that $\bndry X$ is a compact metric space, which implies that $\bndry X$ is separable. Subspaces of separable metric spaces are separable. So, $C$ is separable. For pseudometric spaces separability and second countability are equivalent (see \cite{Wil1} Theorem 16.11), so $C$ is second countable. $Q$ is the continuous open image of a second countable space, therefore $Q$ is second countable (see \cite{Wil1} Theorem 16.2(a)). \end{proof}

\medskip
\begin{lemma}
$Q$ is $T_1$.
\end{lemma}

\begin{proof} A topological space is $T_1$ iff each one point set is closed \cite{Wil1}.
Let $[x]\in Q$. As $Q$ is the quotient of a proper group action $q^{-1}\big([x]\big)$ is a discrete set of points, this implies that $q^{-1}\big([x]\big)$ is closed in $C$. Quotients by group actions are open maps, so $q$ is a surjective open map. Therefore $q(C\setminus q^{-1}\bigl([x]\bigr)=Q\setminus \{[x]\}$ is open, which implies that $[x]$ is closed.\end{proof}

\medskip
\begin{lemma}
$Q$ is regular.
\end{lemma}

\begin{proof} It suffices to show that for each open set $U\in Q$ and $x\in U$ that there exists an open set $V\subset U$ such that $x\in V$ and $\overline{V}\subset U$ (see \cite{Wil1} Theorem 14.3). As $Y$ is locally compact and $C$ is a component of $Y$, $C$ must be locally compact. The continuous open image of locally compact is locally compact, so we have that $Q$ is locally compact. Let $U$ be a neighborhood of $x$ in $Q$. Then by local compactness for any $U$ neighborhood of $x$ in $Q$ there exists an open set $V\subset U$ such that $x\in V$ and $\overline{V}\subset U$. \end{proof}

\medskip

\begin{theorem}
$Q$ is a Peano continuum.
\end{theorem}

\begin{proof}
We have shown that $Q$ is metrizable and $Q$ is compact by definition. $C$ is connected. So, by continuity of $q$, we have that $Q$ is connected. $C$ locally connected and $q$ is a local homeomorphism, so $Q$ is locally connected.\end{proof}

\medskip
Thus, by Bing's theorem we have that $Q$ is a geodesic metric space. Defining $H$ as in the previous two sections we may use the construction mentioned at the beginning of this section to obtain:

\begin{proposition}
There exists an $H$-equivariant metric on $C$.
\end{proposition}

From Corollary \ref{corollary:finitelymany} we know that $\C$ consists of only finitely many components each stabilized by $H$. Thus by defining distance to be the same in each component and the distance between points in different components to be infinite we may prove the following corollary:

\begin{corollary}
There exists an $H$-equivariant metric on $Y$.
\end{corollary}

\medskip
We conclude this section with an important corollary that will prove very useful in Section \ref{sec:inbndryF}. Let $\R$ be the relation defined in Section \ref{sec:ProperCocompact}.

\begin{corollary}
\label{corollary:qirelation}
The relation $\R$ is a quasi-isometry relation, i.e. if $(x_1,y_1),(x_2,y_2)\in\R$ then there exist constants $L>0$ and $C\geq 0$ such that
\[
\frac{1}{L}d(x_1,x_2)-C\leq d_Y(y_1,y_2)\leq Ld(x_1,x_2)+C
\] where $d_Y$ is the $H$-equivariant metric on $Y$ given by Corollary 5.10.
\end{corollary}
\begin{proof} We have $H$ acting geometrically on $F$ and $Y$. So there are quasi-isometries $\alpha\colon H\to F$ and $\beta \colon H \to Y$ given by the orbit maps of the action of $H$ on $F$ and $Y$, respectively. Thus we may find a quasi-isometry $\Phi \colon F\to Y$ given by $\beta\circ\alpha^{-1}$. If
$(x_1,y_1),(x_2,y_2)\in\R$, then we know that there exists $L>0$ and $C\geq0$ such that:


\[
\frac{1}{L}d(x_1,x_2)-C\leq d_Y\big(\Phi(x_1),\Phi(x_2)\big)\leq Ld(x_1,x_2)+C
\].

If we can find a constant $D\geq 0$ such that $d_Y\big(\Phi(x_1), y_1\big)< D$ and $d_Y(\Phi(x_2), y_2)< D$, the we will be done. Let $K\subset F$ be a compact set whose $H$-translates cover $F$. We saw in Section 3 that the projection $pr_F$ and $pr_Y$ are proper and equivariant. So, if $h_i\in H$ is such that $x_i\in h_iK$, then $y_i\in h_iK_{\infty}$ for $i\in\{1,2\}$, where $K_{\infty}= pr_Y(pr_F^{-1})(K)$. We need only that $\Phi(x_i)\in h_iK_{\infty}$ for $i\in\{1,2\}$. But, this follows from the fact that $\Phi$ is the composition of an orbit map and the inverse of an orbit map.
\end{proof}



\section{Local cut points which are not in the boundary of a flat}
\label{sec:notinbndryF}

In this section we wish to prove the following:

\begin{proposition}
\label{proposition:Cutptcase1}
Let $\Gamma$ be a one-ended group acting geometrically on a $\CAT(0)$ space $X$ with isolated flats. Suppose $\bndry X$ is not homeomorphic to $S^1$ and let $\xi\in\bndry X$ be such that $\xi$ is not in $\bndry F$ for any $F\in\F$. If $\xi$ is a local cut point, then $\Gamma$ splits over a $2$-ended subgroup.
\end{proposition}

The proof of this proposition relies on Theorem \ref{theorem:splitting in relative} and a result of Hung Cong Tran (Theorem \ref{theorem:TranTheorem}), which provides a strong connection between $\bndry X$ and $\relbndry$ via a quotient map. Let $f\colon\bndry X\to \relbndry$ be this quotient map. To prove Proposition \ref{proposition:Cutptcase1} we need more information about the behavior of the map $f$. The particular question that needs to be addressed is as follows: Let $\xi\in\bndry X$ which is not in the boundary of a flat. If $\xi$ is a local cut point can its image, $f(\xi)$, fail to be a local cut point in $\relbndry$?

To answer this question in the negative we will first need to recall some basic decomposition theory. We refer the reader to \cite{Daverman} for more information on decomposition theory.

 \subsection{Decompositions}
 \label{subsec: Decomposition}

 A {\it decomposition}, $\D$, of a topological space $X$ is a partition of $X$. Associated to $\D$ is the {\it decomposition} {\it space} whose underlying point set is $\D$, but denoted $X/\D$. The topology of $X/\D$ is given by the {\it decomposition} $map$ $\pi\colon X\to X/\D$, $x\mapsto D$, where $D\in\D$ is the unique element of the decomposition containing $x$. A set $U$ in $X/\D$ is deemed open if and only if $\pi^{-1}(U)$ is open in $X$. A subset $A$ of $X$ is called {\it saturated} (or $\D$-saturated) if $\pi^{-1}\big(\pi(A)\big)=A$. The {\it saturation} of $A$, $Sat(A)$, is the union of $A$ with all $D\in\D$ that intersect $A$. The decomposition $\D$ is said to be {\it upper semi-continuous} if every $D\in\D$ is closed and for every open set $U$ containing $D$ there exists and open set $V\subset U$ such that $Sat(V)$ is contained in $U$. $\D$ is called {\it monotone} if the elements of $\D$ are compact and connected.

 A collection of subsets $\mathcal{S}$ of a metric space is called a $null$ $family$ if for every $\epsilon>0$ there are only finitely many $S\in\mathcal{S}$ with $\diam(S)>\epsilon$. The following proposition can be found as Proposition I.2.3 in \cite{Daverman}.

 \begin{proposition}
 \label{proposition:uppersemi}
 Let $\mathcal{S}$ be a null family of closed disjoint subsets of a compact metric space $X$. Then the associated decomposition of $X$ is upper semi-continuous.
 \end{proposition}

 In the isolated flats setting a theorem of Hruska and Ruane \cite{HR1} shows:

\begin{proposition}
\label{proposition:nullfamily}
 The collection ${\bndry F}_{F\in\F}$ forms a null family in $\bndry X.$
\end{proposition}

\medskip
Let $f\colon\bndry X\to \relbndry$ be as above. Note that $f$ is the decomposition map of the monotone and upper semi-continuous decomposition $\D$ of $\bndry X$ where $\D=\bigset{\bndry F}{F\in\F}\cup\bigset{\{x\}}{x\notin\bndry F\spc \text{for all} \spc F\in\F}$. By Proposition 3.5 of \cite{Hau17a} we have:

\begin{lemma}
\label{lemma:local cut pts to local cut pts}
Let $\xi\in \bndry X$ and assume that $\xi\notin\bndry F$ for any $F\in\F$. If $\xi$ is a local cut point, then $f(\xi)$ is a local cut point.
\end{lemma}

\medskip

Now that we know that non-parabolic local cut points in $\bndry X$ get mapped to non-parabolic local cut points in $\relbndry$, the proof of Proposition \ref{proposition:Cutptcase1} follows almost immediately from Theorem \ref{theorem:splitting in relative}.

\medskip

\begin{proof}[Proof of Proposition 6.1] Let $\xi$ be a point in $\bndry X$ which is a local cut point which is not in that boundary of a flat. As $\CAT(0)$ groups with isolated flats are relatively hyperbolic, Proposition \ref{lemma:local cut pts to local cut pts} implies that there is a non-parabolic local cut point in $\relbndry$. Therefore we are done by Theorem \ref{theorem:splitting in relative}.\end{proof}



\section{Local Cut Points in the Boundary of a Flat}
\label{sec:inbndryF}

The goal of this section is to complete the proof of Theorem \ref{theorem:CutptTheorem} by showing that a point $\xi$ in the boundary of a flat cannot be a local cut point. We begin this section by defining basic neighborhoods ``of infinity'' in $Y=\bndry X\setminus \bndry F$ and provide a useful lemma. Then in Section \ref{subsec:bndryFnotlocalcutpt} we develop machinery required to prove that $\xi$ cannot be a local cut point. Throughout this section we will assume that $\bndry X$ is locally connected.

\subsection{Basic Neighborhoods in $Y$}
\label{subsec:NbrhdsinY}

Let $\xi$ be an element of $\bndry F$. Given a neighborhood $V(\xi,n,\epsilon)$ in the bordification of $X$, recall that $V_{\bndry}(\xi,n,\epsilon)$ is the restriction of $V(\xi,n,\epsilon)$ to points of $\bndry X$. Given a  boundary neighborhood $V_{\bndry}(\xi,n,\epsilon)$ we define $V_Y(\xi,n, \epsilon)$ to be the subset $V_{\bndry}(\xi,n,\epsilon)\backslash\bndry F$. Then $V_Y(\xi,n,\epsilon)$ is open in $Y$ with the subspace topology. Although it is somewhat of a misnomer $V_Y(\xi,n,\epsilon)$, will refer to $V_Y(\xi,n,\epsilon)$  as {\it a basic neighborhood of} $\xi$ {\it in} $Y$. Notice that these sets $V_Y(\xi,n,\epsilon)$ form a basis in the sense that given any open set $A$ consisting of an open neighborhood of $\xi$ in $\bndry X$ intersected with $Y$ we may find a $k>0$ large enough so that $V_Y(\xi, k,\epsilon)\subset A$. Lastly, the set $V(\xi, n, \epsilon)\cap F$ will be referred to as a {\it flat neighborhood } of $\xi$ and denoted $V_F(\xi,n,\epsilon)$. When there is no ambiguity about the parameters $n$ we will simply write $V_{\bndry}$, $V_Y$, and $V_F$. The following is a consequence of Lemmas \ref{lemma:bound1} and \ref{lemma:bound2}:

\begin{lemma}
\label{lemma:flat nbrhd for V}
Suppose $\epsilon>0$, $\xi\in\bndry F$, and $L>0$ is the quasiconvexity constant given by Lemma \ref{corollary:quasiconvex}. There is a $\delta=\delta(L,\epsilon)>\epsilon$ such that for any $n>0$ and $\eta\in V_Y(\xi,n, \epsilon)$ if $r\in\perp(F)$ is the orthogonal representative of $\eta$, then $r(0)\in V_F(\xi,n,\delta)$.
\end{lemma}

\subsection{$\xi\in\bndry F$ cannot be a local cut point}
\label{subsec:bndryFnotlocalcutpt}

Recall that a point $\xi\in\bndry X$ is a {\it local cut point} if $X\setminus\{\xi\}$ is not one-ended. A path connected metric space is one-ended if for each compact $K$ there exists a compact $K'$ such that points outside of $K'$ can be connected by paths outside of $K$. In other words,  to show that $\xi$ is not a local cut point we need to show that for any neighborhood $U_{\bndry}$ of $\xi$, there exists a neighborhood $V_{\bndry}$ of $\xi$ such that all points of $V\backslash\{\xi\}$ can be connected by paths in $U_{\bndry}\backslash\{\xi\}$. Intuitively the idea is to show that we may connect two points close to $\xi$ up by a path which does not travel ``too far'' into $Y$.

\medskip
\medskip

In Section \ref{sec:GeometricAction} we saw that $Y$ admits a geometric action by $H=\bbz^n$; moreover, by  Proposition  \ref{proposition:fullrank} and \ref{lemma:Hperiodic} we know that $Y$ consists of finitely many components whose stabilizers are $\bbz^n$ subgroups of full rank. So the components of $Y$ coarsely look like $\bbz^n$ and this particularly nice structure will help us control the length of paths in $\bndry X$ near $\xi$. 

 The majority of the arguments in this section only concern the action of $\bbz^n$ on a single connected component; therefore, we may assume for now that $Y$ consists of a single connected component. A reader only interested in the proof of Theorem \ref{theorem:MainTheorem} may wish to focus on the simple case when $n=2$, as this is an intuitively simpler case.
 
 Also, recall that our main concern is $1$-dimensional boundaries so the reader may wish to think of $n$ as being equal to $2$ for intuitive purposes; however, the following arguments do not require that assumption.

\begin{lemma}
\label{lemma:pathdisc}
Let $D>0$ and $y\in Y$. Then there exists an $M>0$ such that points of the ball $B(y,D)$ may be connected by paths in $B(y, M)$. 
\end{lemma}

\begin{proof}If $\overline{B}(y, D)$ is connected, then $M=D$ and we are done. So, assume $\overline{B}(y, D)$ is not connected. Then $\overB(y, D)$ must contain only finitely many path components, contradicting local connectedness of $\bndry X$. Let $\mathcal{A}$ be the set of components of $\overB(y, D)$ and assume $|\mathcal{A}|=n$. Then for any pair of components $A_1,A_2\in\mathcal{A}$ with $A_1\neq A_2$ we may find a path $\gamma$ in $Y$ with $\gamma(0)\in A_1$ and $\gamma(1)\in A_2$. Let $P$ be the collection of all such paths. Then $|P|=\binom{n}{2}<\infty$. For every $\gamma\in P$ we have that $\diam(\gamma)<\infty$, so set $N=\max\bigset{\diam(p)}{p\in P}$. Then $\overB(y,D)\cup P$ is connected and has diameter $\leq D+N$. Set $M=D+N$.\end{proof}

\medskip
\begin{corollary}
\label{corollary:transofpathdisc}
Let $D>0$, and $y_0\in Y$. Then there is an $M>0$ such that for any $y\in \orb_H(y_0)$ $B(y, D)$ is path connected inside of $B(y, M)$.
\end{corollary}
\medskip

\begin{proof} This follows immediately from the lemma and the geometric action of $H$ on $Y$.\end{proof}

In the remainder of this section $M$ shall refer to the constant found in Corollary \ref{corollary:transofpathdisc}.

\begin{lemma}
\label{lemma:path lemma}
Let $V_Y=V_Y(\xi, n,\epsilon)$ be a basic neighborhood of $\xi$ in $Y$. Then there exists a metric neighborhood $\N(V_Y)$ of $V_Y$ in $Y$ such that points in $V_Y$ can be connected by paths inside $\N(V_Y)$.

\end{lemma}

\begin{proof} Let $\mu, \nu\in V_{\bndry}$ and $u(0),v(0)$  be the base points of elements $u$ and $v$ of $\perp(F)$ representing $\mu$ and $\nu$ respectively. By Lemma \ref{lemma:flat nbrhd for V} there exists a $\delta>0$ such that $u(0)$ and $v(0)$ are in $V_F(\xi,n,\delta)$. As $V_F(\xi,n,\delta)$ is a sector of an embedded Euclidean plane, there exists a path $\gamma$ in $V_F(\xi,n,\delta)$ connecting $u(0)$ and $v(0)$. Define $A>0$ to be the constant from the definition of the relation $\R$ in Section \ref{sec:ProperCocompact}. We may find a finite sequence of points $(a_i)_{i=0}^n$ contained in the image of $\gamma$ with $a_0=u(0)$, $a_n=v(0)$, and such that $\gamma$ is contained in $\bigcup_{i=0}^{n}\overline{B}(a_i, A)$. By choice of $A$, for each $a_i$ we may find $c_i\in\perp(F)$ with $c_0=a_0$ and $c_n=a_n$, and such that $d_X\big(c_i(0),a_i\big)<A$. This implies that $d_X\big(c_i(0),c_{i+1}(0)\big)<4A$.

The boundary points $c_i(\infty)$ need not be in $V_Y$, but using an argument similar to that of Lemma \ref{lemma:bound2} one sees that they are in $V_Y'=V_Y(\xi, n,\delta+A+K)$ for some $K$. From Lemma \ref{lemma:flat nbrhd for V} we have that $\delta>\epsilon$, so we see that $V_Y\subset V_Y'$.

Recall that Corollary \ref{corollary:qirelation} gives an $(L,C)$-quasi-isometry associated to $\R$, which implies that $d_Y\big(c_i(\infty),c_{i+1}(\infty)\big)< L(4A)+C$. Fixing a base point $y_0$ in $Y$ we know that $H\cong\bbz^n$ acts cocompactly on $Y$, so there exists a constant $J$ and points $\{y_1,...,y_2\}\subset \orb_H(y_0)$ such that for every $i$ we have $d_Y\big(y_i,c_i(\infty)\big)\leq J$. Setting $D=L(4A)+C+J$, we see that the neighborhoods $\overline{B}(y_i, D)$ form a chain from $\mu=c_1(\infty)$ to $\nu=c_n(\infty)$. Corollary \ref{corollary:transofpathdisc} tells us that we may find a constant $M>0$ such that $\bigcup\overline{B}(y_i,D)$ is connected by paths in $N_M\big(\bigcup\overline{B}(y_i,D)\big)$. Therefore, $\mu$ and $\nu$ are connected by a path in $\N(V_Y)=N_M(V_Y')$.
\end{proof}

\begin{figure}[h!]
\includegraphics[scale=.65]{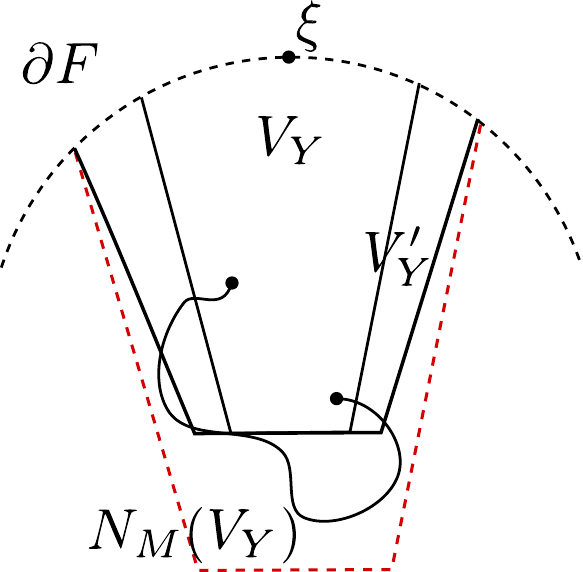}
\captionof{figure}{This figure depicts the neighborhood $\N(V_Y)=N_M(V_Y')$.}
\label{fig: NeigborhoodNV}
\end{figure}

Although it is not truly a neighborhood (in the sense that it is not open in $\bndry X$), we will use $\N(V_{\bndry})$ to denote $\N(V_Y)\cup V_{\bndry}$. In other words, $\N(V_{\bndry})$ is $N(V_Y)$ with $V_{\bndry}\cap\bndry F$ attached.

\medskip
\begin{corollary}
\label{corollary: connecting points near xi}
Let $V_Y$ be a basic neighborhood of $\xi$ in $Y$. Then any two points in $V_{\bndry}\setminus\{\xi\}$ can be connected by paths in $\N(V_{\bndry})\setminus\{\xi\}$.
\end{corollary}

\begin{proof} We have three 3 cases to check. First assume that $\mu,\nu\in V_Y$. Then by \ref{lemma:path lemma} we have that $\mu$ and $\nu$ can be connected by a path in $\N(V_Y)$.

Second, suppose that $\mu,\in V_Y$ and $\nu\in\bndry (F\cap V_{\bndry})\setminus\{\xi\}$. We know that $\bndry X$ is locally path connected. As $V_{\bndry}$ is a basic neighborhood of $\xi$ in $\bndry X$, we may find a path connected basic $\bndry X$-neighborhood $U$ of $\nu$ in $V_{\bndry}\setminus\{\xi\}$. By Corollary \ref{corollary:limitsetcorollary} and choice of $U$ we have that $U\cap V$ is non-empty, we may find a point $\rho\in U\cap V_Y$ which is connected to $\nu$ by a path in $U$. Thus we may apply the first case to connect $\mu$ to $\rho$ by a path in $\N(V_Y)$. By concatenating these paths we complete case two.

Lastly, if $\mu$ and $\nu$ are both in $\bndry (F\cap V_{\bndry})\setminus\{\xi\}$ we may pick a point in $V_Y$ and apply the second case twice.\end{proof}

\medskip


\begin{figure}[h!]
\includegraphics[scale=.6]{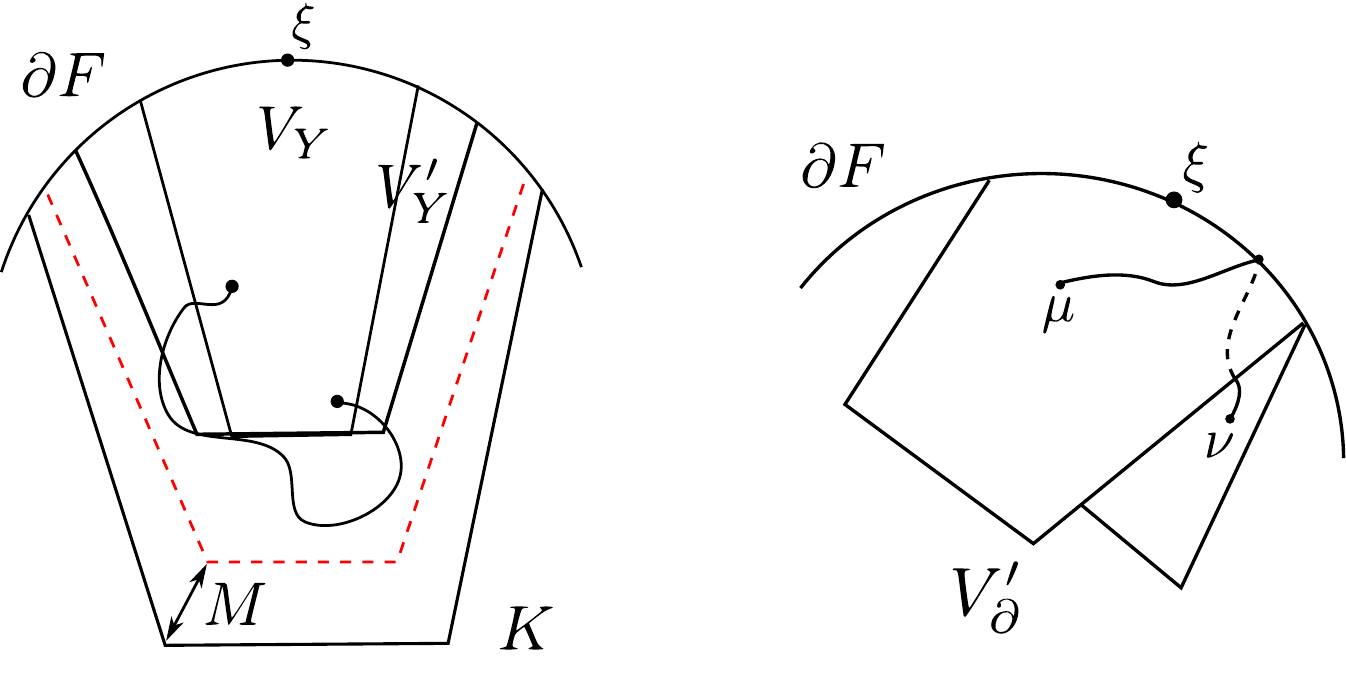}
\captionof{figure}{The figure on the left illustrates the $M$-neighborhood of the compact set $K$ and a path contained in the compliment of its closure connecting two points of $V_Y$. The figure on the right depicts two points near $\xi$ in different components connected by a path in $\bndry X$.}
\label{fig: not local cut point}
\end{figure}

We now disregard the hypothesis that $Y=\bndry X\setminus\bndry F$ consists of a single connected component and prove the main result of this section.

\begin{theorem}
\label{theorem:flat not cut point}
A point $\xi$ in the boundary of a flat cannot be a local cut point.
\end{theorem}

\begin{proof}
Let $K$ be a compact subset of $\bndry X\setminus \{\xi\}$. Recalling the discussion at the beginning of this subsection, we need to find a compact set $L$ such that points outside of $L$ can be connected by paths outside of $K$. Let $\epsilon,\delta,A$ and $K$ be as in Lemma \ref{lemma:path lemma} and set $\kappa=\delta+A+K$. The sets $V_{\bndry}(\xi,n,\kappa)$ form a neighborhood base so we may find an $n>0$ large enough so that \[V_Y(\xi,n,\kappa)\subset V_{\bndry}(\xi,n,\kappa)\subset Y\setminus\overline{N}_M(K),\] where $M>0$ is the constant found in \ref{lemma:path lemma}. Let $V_{\bndry}'=V_{\bndry}(\xi,n,\kappa)$. Notice that $\epsilon<\kappa$ implies that $V_{\bndry}=V_{\bndry}(\xi,n,\epsilon)$ is contained in $V_{\bndry}'$ and consequently $Y\setminus\overline{N}_M(K)$ (see Figure \ref{fig: not local cut point}). Define $L=\bndry X\setminus V_{\bndry}$.

We need that that points in $V_{\bndry}\setminus\{\xi\}$ can be connected by paths outside of $K$. Let $\mu, \nu\in V_{\bndry}\setminus\{\xi\}$. 

If $\mu$ and $\nu$ are in $(\bndry F\cap V_{\bndry})\setminus\{\xi\}$ or in the same component of $V_Y=V_Y(\xi,n,\epsilon)$, then Corollary \ref{corollary: connecting points near xi} tells us that they can be connected a path which misses $K$. So, if $\mu$ and $\nu$ are in different components of $V_Y$ we may connect them by passing through a point of $(\bndry F\cap V_{\bndry})\setminus\{\xi\}$ (see Figure \ref{fig: not local cut point}). Thus, $\mu$ and $\nu$ can be connected by path outside of $K$.\end{proof}

\medskip

Combining this theorem with Proposition \ref{proposition:Cutptcase1} we have completed the proof of Theorem \ref{theorem:CutptTheorem}.

\section{Proof of the Main Theorem}
\label{section: Main Theorem}

The goal of this section is to prove Theorem \ref{theorem:MainTheorem}, but first we review a few facts about the dynamics of the action of $\Gamma$ on $\bndry X$.

\subsection{Tits Distance and the Dynamical Properties of $\Gamma\lacton\bndry X$}

Recall that a {\it line} in $X$ is a 1-flat. A {\it flat half\,plane} is a subspace of $X$ isometric to the Euclidean half-plane, i.e $\{y\geq 0\}$ in the $(x,y)$-plane. A line, $L$, is called {\it rank one} if it does not bound a flat half-plane. By $L(\xi,\eta)$ we denoted the line with endpoints $\xi$ and $\eta$ in $\bndry X$. $L$ is said to be {\it $\Gamma$-periodic} if there is an arc-length parametrization of $\sigma$ of $L$, an element $\gamma\in\Gamma$, and a constant $\alpha>0$ such that $\gamma\sigma(t)=\sigma(t+\alpha)$ for all $t\in\bbr$.

We will need the following two results regarding rank one geodesics and dynamics of the action on the boundary.  The first can be found in Section III.3 of \cite{Ballman} and the second may be found as Proposition 1.10 of \cite{BB1}.

\begin{lemma}
\label{lemma: rankonedynmaics}
Suppose $L$ is an oriented rank one line shifted by an axial isometry $g$. Let $\xi$ and $\eta$ be the end points of $L$. Then for all neighborhoods $U$ of $\xi$ and $V$ of $\eta$ in $\overline{X}$ there exists $n\geq 0$ such that $g^k(\overline{X}\backslash V)\subset U$ and $g^{-k}(\overline{X}\backslash U)\subset V$ for every $k\geq n$.
\end{lemma}

\begin{proposition}
\label{proposition: BB}
Suppose that the limit set $\Lambda \subset \bndry X$ is non-empty. Then the following are equivalent:

\begin{enumerate}[(i)]
\item $X$ contains a $\Gamma$-periodic rank one line.
\item For each $\xi\in\Lambda$ there is a $\eta\in\Lambda$ with $d_T(\xi,\eta)>\pi$, where $d_T$ is the Tits metric.
\item There are points $\xi,\eta\in\Lambda$ with $d_T(\xi,\eta)>\pi$, where $\eta$ is contained in some minimal subset of $\bndry X$.

\end{enumerate}
\end{proposition}

\medskip
  Recall that the {\it Tits metric} is the length metric associated to the angular metric on $\bndry X$. We refer the reader to \cite{BH1} Chapter II.9 for the required background. Though not stated in Proposition \ref{proposition: BB}, it is clear from the proof provided Ballmann and Buyalo \cite{BB1} that the end points of the $\Gamma$-periodic rank one line can be found arbitrarily close to the points $\xi$ and $\eta$. This is precisely the way in which Proposition \ref{proposition: BB} will be used below. We will also need the following:
  \medskip

  \begin{lemma}
\label{lemma: minimalaction}
Let $\Gamma$ be a one-ended group acting geometrically on a $\CAT(0)$ space with isolated flats. The action of $\Gamma$ on $\bndry X$ is minimal.
\end{lemma}

\begin{proof} Assume not, then $\bndry X$ contains a closed $\Gamma$-invariant set $M$. If $f\colon\bndry X \to \relbndry $ is the equivariant quotient map defined in Theorem \ref{theorem:TranTheorem}, then $f$ is a closed map by Proposition 1 of \cite{Daverman}. Thus $f(M)$ is closed and $\Gamma$-invariant. From
\cite{Bow1} we have that the action of $\Gamma$ on $\relbndry$ is minimal. Thus, if $f(M)$ is a proper subset of $\relbndry$ we will obtain a contradiction.

As $M$ is properly contained in $\bndry X$ we may find a neighborhood $U$ of $M$ which is properly contained in $\bndry X$. By upper semi-continuity of the decomposition (see Proposition \ref{proposition:uppersemi}) we have the $Sat(M)\subseteq U$. Thus $\bndry X\setminus Sat(M)\neq \emptyset$, which implies that $f(M)$ is properly contained in $\relbndry$. \end{proof}

\begin{proposition}
\label{proposition: movingaround}
Let $K$ be a proper closed subset of $\bndry X$, then for any $U$ open set in $\bndry X$ we may find a homeomorphic copy $K'$ of $K$ such that $K'\subset U$.
\end{proposition}

\begin{proof} Let $U$ be and open subset of the boundary and let $\rho\in\bndry X\setminus K$ be a conical limit point. In Theorem 5.2.5 of \cite{HK1} Hruska and Kleiner show that components of $\bndry_T X$ are boundary spheres ${\bndry F}_{F\in \F}$ and isolated points. So let $\eta$ be another point in $\bndry X\setminus K$, then $d_T(\rho,\eta)>\pi$. Choose any neighborhoods $V$ of $\rho$ and $W$ of $\eta$ in $\bndry X\setminus K$. Then Proposition \ref{proposition: BB} implies that we may find a periodic rank one line $L$ such that the ends $L(\infty)$ and $L(-\infty)$ are in $V$  and $W$ respectively. We may then apply Lemma \ref{lemma: rankonedynmaics} to find a homeomorphic copy of $K$ in $W$ (or $V$).

By Lemma \ref{lemma: minimalaction} we have that the action of $\Gamma$ on the boundary is minimal, which implies we have that $\orb_{\Gamma}(\eta)$ is dense in $\bndry X$. Thus there exists a $\gamma\in\Gamma$ such that $\gamma\eta\in U$. Choosing $W$ small enough we have that $\gamma(W)\subset U$. As $\gamma$ is a homeomorphism $U$ contains a copy of K. \end{proof}

\medskip
We now prove the main theorem.
\medskip

\begin{proof}[Proof of Theorem \ref{theorem:MainTheorem}:] Using the toplogical characterizations of the Menger curve and Sierpinski carpet we provide a proof similar to that of Kapovich and Kleiner in Section 3 of \cite{KK}.

By hypothesis and Corollary \ref{corollary: HRcorollary}, we have that $\bndry X$ is connected, locally connected, and $1$-dimensional. Theorem \ref{theorem:CutptTheorem} gives that if $\bndry X$ has a local cut point, then $\bndry X$ is homeomorphic to $S^1$ or $\Gamma$ splits over a $2$-ended subgroup. Assume that $\bndry X$ does not have a local cut point.

The boundary of $X$ is planar, or it is not. If $\bndry X$ is planar, then it is a Sierpinski carpet by the characterization of Whyburn \cite{Whyburn}. So, assume that $\bndry X$ is non-planar. Claytor's embedding theorem \cite{Cla} then implies that $\bndry X$ contains a non-planar graph. We may now use Proposition \ref{proposition: movingaround} to show that no non-empty open subset of $\bndry X$ is planar. Thus $\bndry X$ must be a Menger curve by the topological characterization due to Anderson \cite{A58a,A58b}.\end{proof}

\section{Non-hyperbolic Coxeter groups with Sierpinski carpet boundary}
\label{sec: Sierpinski}
In this section we give sufficient conditions for the boundary of a Coxeter group with isolated flats to have a Sierpinski carpet boundary. This result is an easy consequence of Theorem \ref{theorem:MainTheorem} and results of \Swiatkowski\, \cite{Swi16}.

A {\it Coxeter system} is a pair $(W,S)$ such that $W$ is a finitely presented group with presentation $\presentation{S}{R}$ with 
\[ R=\bigset{s^2}{s\in S}\cup\bigset{ (s,t)^{m_{st}}}{s,t\in S,\, m_{st}\in\{2,3,...\infty\}\,\text{and}\, m_{st}=m_{ts}},\] and $m_{st}=\infty$ means that there is no relation between $s$ and $t$.

The {\it nerve} $L=L(W,S)$ of the Coxeter system $(W,S)$ is a simplicial complex whose $0$-skeleton is $S$ and a simplex is spanned by a subset $T\subset S$ if and only if the subgroup generated by $T$ is finite. The {\it labeled nerve} $L^{\bullet}$ of $(W,S)$ is the nerve $L$ with edges $(s,t)$ in the $1$-skeleton of $L$ labeled by the number $m_{st}$. A {\it labeled suspension} in $L^{\bullet}$ is a full subcomplex $K$ of $L$ isomorphic to the simplicial suspension of a simplex, $K=\{s,t\}\ast\sigma$, such that any edge in $K$ adjacent to $t$ or $s$ has edge label $2$. The labeled nerve is called {\it inseparable} if it is connected, has no separating simplex, no separating vertex pair, and no separating labeled suspension. The labeled nerve $L^{\bullet}$ is called a {\it labeled wheel} if $L$ is the cone over a triangulation of $S^1$ with cone edges labeled by $2$.

Associated to any Coxeter system $(W,S)$ is a piecewise Euclidean $\CAT(0)$ space called the {\it Davis complex} $\Sigma=\Sigma(W,S)$. The group $W$ acts geometrically on $\Sigma$ by reflections. Caprace \cite{Cap09} has completely determined when the Davis complex has isolated flats.


\begin{proof}[Proof of Theorem \ref{theorem: IFSwiatkowski}:]
Assume the hypotheses. In Lemmas 2.3, 2.4, and 2.5 of \cite{Swi16} \Swiatkowski\, shows that $\bndry\Sigma$ is  connected, planar, and $1$-dimensional.  Lastly inseparability of $L^{\bullet}$ implies that $W$ does not split over a virtually cyclic subgroup \cite{MT09, Swi16}, thus Theorem \ref{theorem:MainTheorem} implies that $\bndry\Sigma$ must be a circle or a Sierpinski carpet.

If $W$ is hyperbolic, then \Swiatkowski\, \cite{Swi16} shows that $\bndry\Sigma$ cannot be homeomorphic to $S^1$. Assume that $\bndry\Sigma$ is homeomorphic to $S^1$ and $W$ is not hyperbolic. Then $\Sigma$ contains a flat $F$; moreover, $F$ must be the only flat. Thus $W$ is a $2$-dimensional Euclidean group. Because the nerve $L$ is planar, $L$ must be a wheel or a triangulation of $S^1$, a contradiction.

\end{proof}

\section{Non-hyperbolic Groups with Menger Curve Boundary}
\label{sec: Tripleexample}
 In the hyperbolic setting groups with Menger curve boundary are quite ubiquitous. It is a well known result of Gromov \cite{Gro87} that with overwhelming probability random groups are hypeberbolic; subsequently, Dhamani, Guirardel, and Przytycki \cite{DGP11} have shown that with overwhelming probability random groups also have Menger curve boundary. In stark contrast no example of a non-hyperbolic group with Menger curve boundary can presently be found in the literature, leading Kim Ruane to pose the challenge of finding finding a non-hyperbolic group with Menger curve boundary.

Prior to Theorem \ref{theorem:MainTheorem} there were no known techniques for developing examples of such a group. The author claims that one example is the fundamental group of the space obtained by gluing three copies of a finite volume hyperbolic 3-manifold with totally geodesic boundary together along a torus corresponding to a cusp. This particular example was suggested to the author by Jason Mannning, and a detailed proof is to be provided in \cite{HW17}. The author believes that many examples of non-hyperbolic groups with Menger curve boundary may now be constructed in a similar fashion.

\bibliographystyle{plain}
\bibliography{mybib}{}

\end{document}